\documentclass[12pt]{article}

\usepackage[cp1251]{inputenc}
\usepackage{amsmath}
\usepackage{amsthm}
\usepackage{amsxtra}
\usepackage{amsfonts,amssymb}

\newtheorem{Th}{Theorem}
\newtheorem{Prop}[Th]{Proposition}
\newtheorem{Lm}[Th]{Lemma}
\newtheorem{Co}[Th]{Corollary}

\theoremstyle{definition}
\newtheorem{Def}[Th]{Definition}

\author{Dudko A.}
\title{Characters on the Full Group of the Odometer.}
\date{}
\begin{document}
\maketitle
\begin{abstract} Let $X$ be the space of all infinite $0,1$-sequences
and $T$ be the odometer on $X$. In this paper we introduce a dense
subgroup $S(2^\infty)$ of the full group $[T]$ and describe all
indecomposable characters on $S(2^\infty)$. As result we obtain a
description of indecomposable characters on $[T]$.
\end{abstract}
\section{Introduction.}
Let $X=\prod_1^\infty\{0,1\}$ be the space of all infinite
$0,1$-sequences with the product topology and $T$ be the odometer on
$X$: $Tx=x+\underline{1}$. By definition, the {\it full group} of
the automorphism $T$ is the group $[T]$ of all Borel automorphisms
$S$ of $X$, such that $Sx\subset O_T(x)$ for all $x\in X$, where
$O_T(x)$ is the orbit of $x$. Let $\mu=\nu^{\otimes\infty}$ be the
standard product measure on $X$, where $\nu(\{0\})=\nu(\{1\})=1/2$.
The group $[T]$ is a topological group with the {\it uniform}
topology, given by the norm $\left\|S_1-S_2\right\|=\mu(\{x:S_1x\neq
S_2x\})$. In this paper we obtain a description of all indecomposble
characters on $[T]$.

 Let $X_n$ be the set of all $0,1$-sequences of length $n$. Denote
$S(2^n)$ the group of all permutations on $X_n$. Elements of
$S(2^n)$ are arbitrary bijections $X_n\rightarrow X_n$. The group
$S(2^n)$ acts naturally on $X$:$$s\in S(2^n):X\rightarrow X,\;\;
s((x,a))=(s(x),a)\;\;\text{for any}\;\;x\in X_n,a\in X.$$ Denote
$S(2^\infty)=\cup_{n\in \mathbb{N}} S(2^n)$. Then $S(2^\infty)$ is a
dense subgroup in $[T]$. It follows, that for each continuous factor
representation $\pi$ of $[T]$ the restriction of $\pi$ onto
$S(2^\infty)$ generates the same $W^\ast$ algebra, therefore, also
is a factor representation. In this paper we describe all
indecomposable characters on $S(2^\infty)$. It turns out, that
indecomposable characters on $S(2^\infty)$ have very simple
structure and each indecomposable character on $S(2^\infty)$ gives
rise to an indecomposable character on $[T]$.

The group $S(2^\infty)$ is a {\it parabolic} analog of the infinite
symmetric group $S(\infty)$. Another parabolic analog of $S(\infty)$
is the group $R$ of {\it rational rearrangements of the segment}
(see \cite{Gor}). In \cite{Gor} E. Goryachko studied $K_0$-functor
and characters of the group $R$. Indecomposable characters on the
infinite symmetric group were described by E. Thoma in \cite{Thoma}.
In \cite{VK0} and \cite{VK1} A. Vershik and S. Kerov developed the
asymptotic theory of characters on $S(\infty)$. In \cite{O1} and
\cite{O2} G. Olshanski developed the semigroup approach to
representations of groups, connected to $S(\infty)$. Using the
semigroup approach, A. Okounkov found a new proof of the Thoma's
result (see \cite{Ok2}). In this paper we use the approach of
Olshanski and Okounkov.

The author is grateful to Nessonov N.I. for the statement of the
problem and useful discussions.

Now we remind some definitions from the representation theory.
\begin{Def}
A character on a group $G$ is a function $\chi:G\rightarrow
\mathbb{C}$, satisfying to the following properties:
\begin{itemize}
\item[1)] $\chi(g_1g_2)=\chi(g_2g_1)$ for any $g_1,g_2\in G$;
\item[2)] the matrix
$\left\{\chi\left(g_ig_j^{-1}\right)\right\}_{i,j=1}^n$ is
nonnegatively defined for any $n$ and $g_1,\ldots,g_n\in G$;
\item[3)] $\chi(e)=1$.
\end{itemize} A character $\chi$ is called indecomposable, if it
can't be represented in the form $\chi=\alpha
\chi_1+(1-\alpha)\chi_2$, where $0<\alpha<1$ and $\chi_1,\chi_2$ are
distinct characters.
\end{Def}
For a unitary representation $\pi$ of a group $G$ denote
$\mathcal{M}_\pi$ the $W^{*}$-algebra, generated by the operators of
the representation $\pi$. By definition, the commutant $S'$ of a set
$S$ of operators in a Hilbert space $\mathcal{H}$ is the algebra
$S'=\{A\in B(\mathcal{H}):AB=BA\text{ for any }B\in S\}$.
\begin{Def} A representation $\pi$ of a group $G$ is called a {\it
factor representation}, if the algebra $\mathcal{M}_\pi$ is a
factor, that is $\mathcal{M}_\pi \cap \mathcal{M}_\pi'=\mathbb{C}I$.
\end{Def} The indecomposable characters on a group $G$ are in one to
one correspondence with the {\it finite type} factor representations
of $G$. Namely, starting with an indecomposable character $\chi$ on
$G$ one can construct a triple
$\left(\pi_\chi,\mathcal{H}_\chi,\xi_\chi\right)$, called the
Gelfand-Naimark-Siegal construction. Here $\pi_\chi$ is a finite
type factor representation, acting in the space $\mathcal{H}_\chi$,
and $\xi_\chi$ is a unit vector in $\mathcal{H}_\chi$, such that
$\chi(g)=(\pi_\chi(g)\xi_\chi,\xi_\chi)$ for any $g\in G$. Note,
that the vector $\xi_\chi$ is cyclic and separating for the algebra
$\mathcal{M}_\pi$.

For $n\in \mathbb{N}$ denote the inclusion
$\mathfrak{i}_n:S(2^\infty)\hookrightarrow S(2^\infty)$ as follows:
\begin{eqnarray} \mathfrak{i}_n(s)((x,y))=(x,s(y)),\;\;\text{for
any}\;\;x\in X_n,y\in X.
\end{eqnarray}
Put $S_n(2^\infty)=\mathfrak{i}_n(S(2^\infty))$. Note, that for any
$s_1\in S(2^n), s_2\in S_n(2^\infty)$ one has $s_1s_2=s_2s_1$. The
following property is known as {\it multiplicativity}.
\begin{Prop}\label{intro mult} A character $\chi$ on $S(2^\infty)$ is indecomposable
iff $\chi(s_1s_2)=\chi(s_1)\chi(s_2)$ for any $n\in\mathbb{N}$ and
$s_1\in S(2^n),s_2\in S_n(2^\infty)$.
\end{Prop} The last proposition can be proven the same way, as
the analogous statement for the indecomposable characters on the
infinite symmetric group (see \cite{Ok2}).
 For $f\in [T]$ denote
$Fix(f)=\{x\in X:f(x)=x\}$. The main results of this paper are the
following two propositions:
\begin{Th}\label{intro main}
A function $\chi$ on $S(2^\infty)$ is an indecomposable character,
if and only if there exists $\alpha\in \mathbb{Z}_+\cup \{\infty\}$,
such that $\chi(s)=\mu(Fix(s))^\alpha$ for any $s\in S(2^\infty)$.
\end{Th}
In the last theorem  we assume $0^0=1,x^\infty=0$ for any $x\in
[0,1)$ and $1^\infty=1$. Note, that $\alpha=0$ and $\infty$
correspond to the trivial and the regular characters.
\begin{Co} A function $\chi$ on $[T]$ is an indecomposable character,
if and only if there exists $\alpha\in \mathbb{Z}_+\cup \{\infty\}$,
such that $\chi(f)=\mu(Fix(f))^\alpha$ for any $f\in [T]$.
\end{Co}

\section{Construction of representations.}
In this section we give a construction of $II_1$
factor-representations of $[T]$. Denote
\begin{eqnarray}Y_n=\{(x,y)\in X\times X:x_k=y_k\;\;\text{for
all}\;\;k>n\},\;\; Y=\cup \,Y_n.\end{eqnarray} For $a,b\in X_n$
introduce cylindrical set \begin{eqnarray}Y_n^{a,b}=\{(x,y)\in
Y_n:x_i=a_i,y_i=b_i\;\;\text{for all}\;\;i\leqslant
n\}.\end{eqnarray} Introduce the measure $\gamma$ on $Y$ by the
formula
\begin{eqnarray}
\gamma\left(Y_n^{a,b}\right)=2^{-n}\;\;\text{for each}\;\;n\in
\mathbb{N}\;\;\text{and}\;\;a,b\in X_n.
\end{eqnarray} Denote the unitary
representation $\pi$ of $[T]$ in $\mathcal{H}=L^2(Y,\gamma)$ by the
formula
\begin{eqnarray}
(\pi(s)f)(x,y)=f(s^{-1}(x),y)\;\;\text{for each}\;\;f\in
\mathcal{H},s\in [T].
\end{eqnarray}
Put $\xi(x,y)=\delta_{x,y}$ and $\chi(s)=(\pi(s)\xi,\xi),s\in [T]$.
Then direct calculations show, that $\chi(s)=\mu(Fix(s))$ for all
$s\in [T]$. In particular, $\chi$ is a central function on $[T]$. It
follows, that $\chi$ is a character. By the proposition \ref{intro
mult}, $\chi$ is an indecomposable character on $S(2^\infty)$. It
follows, that $\chi$ is an indecomposable characters on $[T]$.
Moreover, for any $k\in \mathbb{N}$, considering the triple
$\left(\pi^{\otimes k},\mathcal{H}^{\otimes k},\xi^{\otimes
k}\right)$, we get, that $\chi^k$ is an indecomposable character.

\section{System of orthogonal projections.} Let $\chi$ be an indecomposable
character on the group $S(2^\infty)$. Denote $(\pi,\mathcal{H},\xi)$
the corresponding $GNS$-construction. In this section we find a
system of orthogonal projections in the algebra $\mathcal{M}_\pi$,
satisfying remarkable properties.

First we describe the conjugate classes in $S(2^\infty)$. Let
$g_1,g_2\in S(2^\infty)$. Then there exists $n$, such that
$g_1,g_2\in S(2^n)$. The elements $g_1$ and $g_2$ are conjugate in
$S(2^\infty)$, if and only if $g_1$ and $g_2$ are conjugate in
$S(2^n)$. Remind, that conjugate classes in finite symmetric groups
are parameterized by partitions, made from the lengths of the
cycles.

For subsets $A\subset X_n,B\subset X$ denote $$A\times
B=\{(a_1,\ldots,a_n,b_1,b_2,\ldots):(a_1,\ldots,a_n)\in
A,(b_1,b_2,\ldots)\in B\}\subset X.$$ We will call a subset
$A\subset X$ {\it nice}, if $A=C\times X$ for some $k\in\mathbb{N}$
and $C\subset X_k$. Let $A$ be nice and $m>k$. Denote $s^A_m\in
S(2^\infty)$ as follows:
\begin{eqnarray}\label{asymp snm}
s^A_m(x)=\left\{\begin{array}{ll}x,&\text{if } x\in
A,\\(x_1,\ldots,x_{m-1},1-x_m,x_{m+1},\ldots),&\text{if } x\notin A.
\end{array}\right..
\end{eqnarray}
 Note, that $s^A_m$ affects only $m$-th coordinate of an element of $X$.
 \begin{Lm}\label{asymp proj} For any nice $A\subset X$ there exists the weak limit
 $P^A=w-\lim\limits_{m\rightarrow\infty}
 \pi\left(s^A_m\right)$. The operators $P^A$ are
 orthogonal projections.
 \end{Lm}
 \begin{proof} Consider any elements $g_1,g_2\in S(2^\infty)$. Fix
 $M$, such that $g_1,g_2\in S(2^M)$. One can check, that for
 any $m>M$ the conjugate class of the element $g_1^{-1}s^A_mg_2$
 doesn't depend on $m$. By the centrality of $\chi$, the value
 \begin{eqnarray}\left(\pi\left(s^A_m\right)\pi(g_2)\xi,\pi(g_1)\xi\right)=
 \chi\left(g_1^{-1}s^A_mg_2\right)
 \end{eqnarray} doesn't depend on the choice of $m>M$. By definition of the GNS-construction, the
  vectors $\pi(g)\xi,g\in S(2^\infty)$ are dense in $\mathcal{H}$.
  Therefore, there exists the weak limit $P^A=w-\lim\limits_{m\rightarrow\infty}
 \pi\left(s^A_m\right)$.

 Further, since
 $\pi\left(s^A_m\right)^{*}=\pi\left(\left(s^A_m\right)^{-1}\right)=\pi\left(s^A_m\right)$
 for any $m$, the operator $P^A$ is self-adjoint.
  As follows, $\left(P^A\right)^2$ is a positive operator.  Obviously, $\left\|P^A\right\|\leqslant 1$. One can check,
 that for any $m_1>m_2>m_3$ the elements $s^A_{m_1}s^A_{m_2}s^A_{m_3}$ and $s^A_{m_1}s^A_{m_2}$
 are conjugate. Therefore, \begin{eqnarray*}\left(\pi\left(s^A_{m_1}s^A_{m_2}s^A_{m_3}\right)\xi,\xi\right)=
 \left(\pi\left(s^A_{m_1}s^A_{m_2}\right)\xi,\xi\right).\end{eqnarray*} In
 the limit we get:\begin{eqnarray}\label{asymp eq}\left(\left(P^A\right)^3\xi,\xi\right)=\left(\left(P^A\right)^2\xi,\xi\right)=
 \left\|P^A\xi\right\|^2.\end{eqnarray}
 From the other hand, by the Cauchy-Schwartz inequality,
 \begin{eqnarray*}
\left(\left(P^A\right)^3\xi,\xi\right)\leqslant
\left\|P^A\xi\right\|\left\|\left(P^A\right)^2\xi\right\|\leqslant\left\|P^A\xi\right\|^2.
 \end{eqnarray*} The equality means, that $P^A\xi=c\left(P^A\right)^2\xi$ for some constant $c$. Since $\xi$ is separating,
 the latter means, that $P^A=c\left(P^A\right)^2$. From (\ref{asymp eq}) we get, that $P^A=\left(P^A\right)^2$,
  which finishes the proof.\end{proof}
  Recall, that the unique normalized trace on the algebra $\mathcal{M}_\pi$ is given by
  the formula: $tr(T)=(T\xi,\xi)$.
  \begin{Prop}\label{asymp properties} For any nice $A,B\subset X$ and $C\subset X_n,D\subset X_m$,
  $n,m\in\mathbb{N}$, the following is true:
  \begin{itemize}
  \item[1)] if $s\in S(2^\infty)$, then
  $\pi(s)P^A\pi(s^{-1})=P^{s(A)}$;
  \item[2)] $P^AP^B=P^{A\cap B}$;
  \item[3)] $tr\left(P^{C\times
  D\times X}\right)=tr\left(P^{C\times X}\right)tr\left(P^
  {D\times X}\right)$;
  \item[4)] if
  $\mu(A)\leqslant \mu(B)$, then $tr\left(P^A\right)\leqslant
  tr \left(P^B\right)$.
  \end{itemize}
  \end{Prop}
  In the item $3)$ by $C\times D\times X$ we mean the set of sequences of the
  form $$(c_1,\ldots,c_n,d_1,\ldots,d_m,x_1,x_2,\ldots),$$ where
  $(c_1,\ldots,c_n)\in C$, $(d_1,\ldots,d_m)\in D$ and $x_i\in\{0,1\}$.
  \begin{proof} $1)$ The first property follows immediately from the
  equation $ss_{m}^As^{-1}=s_{m}^{s(A)}$ for large enough $m$.

 $2)$ Let $m_1>m_2>m_3$. It follows from (\ref{asymp snm}), that the
  elements $s_{m_1}^As_{m_2}^Bs_{m_3}^{A\cap
  B}$, $s_{m_1}^As_{m_2}^B$ and
  $s_{m_1}^{A\cap B}$ are in the same conjugate class. Therefore $$\chi\left(s_{m_1}^As_{m_2}^Bs_{m_3}^{A\cap
  B}\right)=\chi\left(s_{m_1}^As_{m_2}^B\right)=\chi\left(s_{m_1}^{A\cap B}\right).$$ When $m_1,m_2,m_3$ go
to infinity,
  we get
  \begin{eqnarray}\label{asymp eq1}tr\left(P^AP^BP^{A\cap
  B}\right)=tr\left(P^AP^B\right)=tr\left(P^{A\cap B}\right).
  \end{eqnarray} From the other hand, by the Cauchy-Schwartz inequality, centrality of $tr$ and lemma \ref{asymp proj},
\begin{eqnarray}tr\left(P^AP^BP^{A\cap
  B}\right)\leqslant tr\left(P^AP^B\right)^{\frac{1}{2}}tr\left(P^{A\cap B}\right)^{\frac{1}{2}}.
  \end{eqnarray} By (\ref{asymp eq1}), the equality holds.
  Therefore, as in the proof, that $P^A$ is an orthogonal
  projection, we get, that $P^AP^B=P^{A\cap B}$.

  $3)$ This property follows from the multiplicativity of $\chi$ (see
prop. \ref{intro  mult}). Indeed, by the definition of operators
  $P^A$ and the proof of the lemma \ref{asymp proj}, one has
  \begin{eqnarray}tr\left(P^{C\times D\times X}\right)=\chi
  \left(s_{n+m+1}^{C\times D\times X}\right).
  \end{eqnarray} The conjugate class of the element $s_{n+m+1}^{C\times
  D\times X}$ contains
  $s_{n+1}^{C\times X} s_{n+m+2}^{X_{n+1}\times D\times X}$. Thus,
  $\chi\left(s_{n+m+1}^{C\times D\times X}\right)=\chi\left(s_{n+1}^{C\times X} s_{n+m+2}^{X_{n+1}\times
  D\times X}\right)$. By the proposition \ref{intro mult},
  \begin{eqnarray}\chi\left(s_{n+1}^{C\times X} s_{n+m+2}^{X_{n+1}\times
  D\times X}\right)=\chi\left(s_{n+1}^{C\times X}\right)\chi\left(s_{n+m+2}^{X_{n+1}\times
  D\times X}\right).\end{eqnarray} To finish the proof, we note, that by
  centrality of $\chi$,
  $\chi\left(s_{n+1}^{C\times X}\right)=tr\left(P^{C\times X}\right)$ and $\chi\left(s_{n+m+2}^{X_{n+1}\times
  D\times X}\right)=tr\left(P^{D\times X}\right)$.

  $4)$ By the property $1)$, without loss
  of generality we may assume, that $A\subset B$. By the property
  $2)$, $P^A\leqslant P^B$. Therefore, $tr\left(P^A\right)\leqslant
  tr \left(P^B\right)$.
  \end{proof}
 \begin{Co}\label{asymp alpha} There exists $\alpha\in \mathbb{R}_+\cup\{\infty\}$, such that for any $n$
  and finite union of cylinders $A\subset X$ one has
  $tr\left(P^A\right)=\mu(A)^\alpha$.
\end{Co}
\begin{proof} We split the proof into three cases, according to the
possible values of $\alpha$ ($0$, $\infty$ or a positive number).

$1.$ First assume, there exists $C\subset X_n$, $C\neq X_n$, such
that $tr\left(P^{C\times X}\right)=1$. Then for any $m$ and any
$D\subset X_m, D\neq \emptyset$ one can find $k$, such that
$\mu\left(C^k\times X\right)\leqslant \mu(D\times X)$. By the
proposition \ref{asymp properties}, $tr\left(P^{D\times
X}\right)\geqslant tr\left(P^{C^k\times X}\right)=1$. Therefore,
$tr\left(P^{D\times X}\right)=1$. Now we only need to check, that
$tr\left(P^{\emptyset}\right)=1$. Since $\xi$ is separating for
$\mathcal{M}_\pi$, it follows, that $P^A=Id$ for any finite union of
cylinders $A\neq \emptyset$. By the property $2)$ from the
proposition \ref{asymp properties}, $P^{\emptyset}=Id$. Thus, in
this case the corollary holds for $\alpha=0$.

$2.$ Now assume, that there exists $C\subset X_n$, $C\neq
\emptyset$, such that $tr\left(P^{C\times X}\right)=0$. Using the
same ideas, as in the case $1$, one can prove, that for any finite
union of cylinders $B\subset X,B\neq X$, one has
$tr\left(P^B\right)=0$. Since $P^X=Id$, $tr\left(P^X\right)=1$. In
this case the corollary holds for $\alpha=\infty$.

$3.$ Now assume, that $0<tr\left(P^A\right)<1$ for any $n$ and any
$A\subset X$, such that $A\neq X$ and $A\neq \emptyset$. It follows
from the property $4)$ of the previous proposition, that
$tr\left(P^A\right)$ depends only on the measure of $A$. Therefore,
there is a function $\varphi$ from the set
$\mathcal{D}=\left\{\frac{p}{2^q}:p,q\in\mathbb{N}, p<2^q \right\}$
of dyadic numbers to the set of positive numbers, such that
\begin{eqnarray}
tr\left(P^A\right)=\varphi(\mu(A))\;\;\text{for any}\;\;A\subset
X,A\neq X,A\neq\emptyset.
\end{eqnarray} By the properties $3),4)$ from the proposition \ref{asymp
properties}, $\varphi$ is a monotone multiplicative homomorphism. It
follows, that there exists $0<\alpha<\infty$, such that
$\phi(d)=d^\alpha$ for any $d\in\mathcal{D}$.
\end{proof}

\section{The proof of the classification theorems.}
\begin{Prop}\label{asymp fixproj} Let $A\subset X$ be nice and $s\in S(2^\infty)$, such
that $A\subset Fix(s)=\{x\in X:s(x)=x\}$. Then $\pi(s)P^A=P^A$.
\end{Prop}
\begin{proof} There exists $n$, such that
$A=C\times X$ for some $C\subset X_n$ and $s\in S(2^n)$.
 Assume first, that $s$ contains only cycles of length $1$ and $2$.
Than the permutations $ss_{m}^A$ and $s_{m}^A$ are conjugate for
large $m$. Therefore,
$\left(\pi(s)P^A\xi,\xi\right)=\left(P^A\xi,\xi\right)=\left\|P^A\xi\right\|^2$.
 Using the Cauchy-Schwartz inequality, we get
 \begin{eqnarray}\left(\pi(s)P^A\xi,\xi\right)=\left(\pi(s)P^A\xi,P^A\xi\right)\leqslant
\left\|\pi(s)P^A\xi\right\|\left\|P^A\xi\right\|
\leqslant\left\|P^A\xi\right\|^2.\end{eqnarray} Since the equality
holds, $\pi(s)P^A\xi=P^A\xi$. Since $\xi$ is separating,
$\pi(s)P^A=P^A$. Now notice, that permutations $s\in S(2^n)$, such
that $A\subset Fix(s)$ and $s$ has only cycles of length 1 and 2,
generate all permutations $w\in S(2^n)$, such that $A\subset
Fix(w)$. This finishes the proof.
\end{proof}
\begin{Co}\label{asymp evencycles}
Let $s\in S(2^\infty)$ have cycles of length 1 and only. Then
$\chi(s)=\mu(Fix(s))^\alpha$, where $\alpha$ is from the corollary
\ref{asymp alpha}.
\end{Co}
\begin{proof} Denote $A=Fix(s)$. Then for large $m$ the elements $s$ and $ss_{m}^A$ are
conjugate. Therefore, using the propositions \ref{asymp fixproj} and
corollary \ref{asymp alpha}, we get
\begin{eqnarray}\chi(s)=tr\left(\pi(s)P^A\right)=tr\left(P^A\right)=\mu(A)^\alpha.
\end{eqnarray}
\end{proof}
\begin{Prop}
For any $s\in S(2^\infty)$ one has $\chi(s)=\mu(Fix(s))^\alpha$,
where $\alpha$ is the number from the corollary \ref{asymp alpha}.
\end{Prop}
\begin{proof}
Let $s\neq e$. Put $A=Fix(s)$. Fix arbitrary $k\in \mathbb{N}$.
There exist permutations $s_1,\ldots,s_k\in S(2^\infty)$, such that
the following is true:
\begin{itemize}
\item[1)] elements $s_i$ are conjugate to $s$;
\item[2)] $Fix\left(s_is_j^{-1}\right)=Fix(s_i)=A$ for any $i\neq j$;
\item[3)] for any $i\neq j$ the element $s_is_j^{-1}$ has cycles of length 1 and 2 only.
\end{itemize} We postpone the proof of the existence of $s_i$ to the
appendix, since this statement is purely combinatoric. Consider the
system of vectors $\eta_i=\left(\pi(s_i)-P^A\right)\xi$. Put
$\eta=\left(\pi(s)-P^A\right)\xi$. It follows from $1)$, that
$(\eta_i,\xi)=(\eta,\xi)$ for any $i$. By the proposition \ref{asymp
fixproj} and property $2)$ of $s_i$, $\pi(s_i)P^A=P^A$. Thus, by the
corollaries \ref{asymp alpha} and \ref{asymp evencycles}, for any
$i\neq j$
\begin{eqnarray}
(\eta_i,\eta_j)=\left(\left(\pi(s_i)-P^A\right)\xi,\left(\pi(s_j)-P^A\right)\xi\right)=
\chi\left(s_is_j^{-1}\right)-tr\left(P^A\right)=0.
\end{eqnarray} Note, that by the same reasons $\|\eta_i\|=\sqrt{1-tr\left(P^A\right)}=\|\eta\|$ for any $i$. Therefore, one has:
\begin{eqnarray}
|(\eta,\xi)|=\frac{1}{k}\left|\left(\sum\limits_{i=1}^k\eta_i,\xi\right)\right|\leqslant
\frac{1}{k}\left\|\sum\limits_{i=1}^k\eta_i\right\|=\frac{\|\eta\|}{\sqrt{k}}.
\end{eqnarray}
Since $k$ is arbitrary, the last inequality means $(\eta,\xi)=0$. It
follows, that $\chi(s)=tr\left(P^A\right)=\mu(A)^\alpha$.
\end{proof}
\begin{Prop} Let $0<\alpha<\infty$. Denote $\chi_\alpha(s)=\mu(Fix(s))^\alpha,s\in S(2^\infty)$.
Assume, that $\chi_\alpha$ is a character. Then $\alpha\in
\mathbb{N}$.
\end{Prop}
\begin{proof} Note first, that $\chi_\alpha$ is indecomposable by the proposition \ref{intro mult}.
Let $(\pi_\alpha,\mathcal{H}_\alpha,\xi_\alpha)$ be the
GNS-construction, corresponding to $\chi_\alpha$.

Let $n\in \mathbb{N}$. Following Okounkov $\cite{Ok2}$, consider the
orthogonal projection
$$Alt(n)=\frac{1}{2^n!}\sum\limits_{s\in
S(2^n)}\sigma(s)\pi_\alpha(s),$$ where $\sigma(s)$ is the sign of
the permutation $s$. One has:\begin{eqnarray}\label{main
sum}\begin{split}0\leqslant(Alt(n)\xi_\alpha,\xi_\alpha)=\frac{1}{2^n!}\sum\limits_{s\in
S(2^n)}\sigma(s)\chi_\alpha(s)=\\ \frac{1}{2^n!}\sum\limits_{s\in
S(2^n)}\sigma(s)\frac{|\{x\in X_n:s(x)=x\}|^\alpha}{2^{\alpha
n}}.\end{split}\end{eqnarray} Calculate the last sum.
 Denote
$\Sigma_k=\sum\limits_{E_k}\sigma(s)$, where $E_k$ is the set of
permutations $s\in S(k)$, such that $s(j)\neq j$ for $1\leqslant
j\leqslant k$. We prove by induction, that
$\Sigma_k=(-1)^{k-1}(k-1)$. Base $k=1,2$ is obvious. Let $k\geqslant
2$. Each element of $E_{k+1}$ can be represented as $(i,k+1)s$,
where $s$ is either any element of $E_k$, or a permutation from
$S(k)$, such that $s(i)=i$ and $s(j)\neq j$ for $j\neq i,1\leqslant
j\leqslant k$. Therefore,
$\Sigma_{k+1}=-k(\Sigma_k+\Sigma_{k-1})=(-1)^kk$. Let $m\in
\mathbb{N}$. One has:
\begin{eqnarray}\label{main sum1}\sum\limits_{s\in
S(m)}\sigma(s)|\{j:s(j)=j\}|^\alpha=\sum\limits_{A\subset
\{1,\ldots,m\}}\sum\limits_{Fix(s)=A}\sigma(s)|A|^\alpha.
\end{eqnarray} For any $j$ there are
$C_m^j$ subsets $A$ of cardinality $m-j$, each of which makes
contribution $\Sigma_j\cdot(m-j)^\alpha=(-1)^{j-1}(j-1)(m-j)^\alpha$
to the sum $(\ref{main sum1})$. Thus, \begin{eqnarray}\label{main
sum2}\sum\limits_{A\subset
\{1,\ldots,m\}}\sum\limits_{Fix(s)=A}\sigma(s)|A|^\alpha=
\sum\limits_{j=0}^mC_m^j(-1)^{j-1}(j-1)(m-j)^\alpha.
\end{eqnarray} From $(\ref{main sum})-(\ref{main sum2})$ for $m=2^n$ one gets:
\begin{eqnarray}\label{main
sum3}C_\alpha(m)=
\sum\limits_{j=0}^mC_m^j(-1)^{j-1}(j-1)(m-j)^\alpha\geqslant 0.
\end{eqnarray} We will show, that for any noninteger $\alpha>0$ there exist $m\in
\mathbb{N}$, such that $C_\alpha(m)<0$.

Note, that for $\alpha=n\in \mathbb{N}$ the last sum can be written
in terms of Stirling numbers of the second type:
\begin{eqnarray*}C_n(m)=m!(S(n,m)+S(n,m-1)),\\ \text{where}\;\; S(n,m)=
\frac{1}{m!}\sum\limits_{j=0}^mC_m^j(-1)^{j}(m-j)^n.
\end{eqnarray*} Remind, that $S(n,m)=0$ for $n<m$ and $S(n,m)>0$ for $n\geqslant m$.
Further, using the binomial rule, we get:
\begin{eqnarray*}C_\alpha(m)=m^\alpha\sum\limits_{j=0}^mC_m^j(-1)^{j-1}(j-1)\left(1-\frac{j}{m}\right)^\alpha=\\
m^\alpha\sum\limits_{j=0}^mC_m^j(-1)^{j-1}(j-1)\sum\limits_{k=0}^\infty
C_\alpha^k\left(-\frac{j}{m}\right)^k=\\ \sum\limits_{k=0}^\infty
(-1)^km^{\alpha-k}C_\alpha^k\sum\limits_{j=0}^mC_m^j(-1)^{j-1}(j-1)j^k.
\end{eqnarray*} Using change $r=m-j$, we get:
\begin{eqnarray*}
\sum\limits_{j=0}^mC_m^j(-1)^{j-1}(j-1)j^k=\sum\limits_{r=0}^mC_m^r(-1)^{m-r-1}(m-r-1)(m-r)^k=\\
(-1)^{m-1}\left((m-1)\sum\limits_{r=0}^mC_m^r(-1)^r(m-r)^k-\sum\limits_{r=0}^mrC_m^r(-1)^r(m-r)^k\right)=\\
m!((m-1)S(k,m)+S(k,m-1)).
\end{eqnarray*}
Finally, we get $$C_\alpha(m)=m!\sum\limits_{k={m-1}}^\infty
(-1)^{k+m-1}m^{\alpha-k}C_\alpha^k((m-1)S(k,m)+S(k,m-1)).$$ Let
$\alpha$ be noninteger. The sign of the expression
$(-1)^kC_\alpha^k$ doesn't depend on $k$ for $k>[\alpha]+1$. It
follows, that $C_\alpha(m)<0$ either for $m=[\alpha]+3$ or for
$m=[\alpha]+4$. This finishes the proof.
\end{proof}
\section{Appendix: existence of $s_i$.}
Here we prove the next combinatorial statement.
\begin{Prop}\label{append si}
Let $s\in S(2^\infty)$. Then for any $r$ there exist permutations
$s_1,\ldots,s_{2^r}\in S(2^\infty)$, such that the following is
true:
\begin{itemize}
\item[1)] elements $s_i$ are conjugate to $s$;
\item[2)] $Fix\left(s_is_j^{-1}\right)=Fix(s_i)=A$ for any $i\neq j$;
\item[3)] for any $i\neq j$ the element $s_is_j^{-1}$ has cycles of length 1 and 2 only.
\end{itemize}
\end{Prop}
We will divide the proof of the last proposition into several
lemmas. For pairwise distinct numbers $k_0,\ldots,k_{l-1}$ denote
$(k_0,k_1\ldots,k_{l-1})$ the cyclic permutation, sending $k_i$ to
$k_{i\, mod\, l}$. In particular, $(k,l)$ stands for the
transposition of $k$ and $l$.
\begin{Lm}\label{append g1}
Let $k>4$ be an odd number. Then for $l\in\{2k-2,2k-4\}$ there exist
permutations $g_1,g_2\in S(l)$, such that each of $g_1,g_2$ has only
cycles of length $k$ or $1$, and $g_1g_2^{-1}$ has only even cycles.
\end{Lm}
\begin{proof}
Let $l=2k-2$. Put
\begin{eqnarray*}g_1=(1,2,\ldots,k)=(1,2)(2,3)\cdots
(k-1,k),\;\;\; g_2=(2k-2,2k-3,\\
\ldots,k-1)=(2k-2,2k-3)(2k-3,2k-4)\cdots(k,k-1).\end{eqnarray*} Then
$g_1$ and $g_2$ are cycles of length $k$ and \begin{eqnarray*}
g_1g_2^{-1}=(1,2)(2,3)\cdots (k-2,k-1)\times (k,k+1)(k+1,k+2)\\
\cdots (2k-3,2k-2)=(1,2,\ldots
,k-1)\times(k,k+1,\ldots,2k-2)\end{eqnarray*} is a product of two
cycles of length $k-1$.

Let $l=2k-4$. Then denote $g_1=(1,2)(2,3)\cdots(k-1,k)$ as above and
$g_2=(2k-4,2k-5)(2k-5,2k-6)\cdots
(k+1,k)\times(k-3,k-2)(k-2,k-1)(k-1,k)$. Again, $g_1$ and $g_2$ are
cycles of length $k$. One has
\begin{eqnarray*}g_1g_2^{-1}=(1,2)(2,3)\cdots (k-4,k-3)\times\\(k,k+1)(k+1,k+2)\cdots(2k-5,2k-4)\end{eqnarray*} is a
product of two cycles of length $k-3$.
\end{proof}
\begin{Co}\label{append Mk} For any $k$ there exists $M=M(k)$, such that for any
$m\geqslant M$ there exist $g_1,g_2\in S(2^m)$, satisfying the
following conditions:\\
$1)$ $g_1$ and $g_2$ have cycles of length, dividing $k$;\\
$2)$ $g_1g_2^{-1}$ has only even cycles.
\end{Co}
\begin{proof}
If $k$ is even, put $M=2$. Let $m\geqslant 2$. Consider $S(2^m)$ as
the group of permutations on $X_m$. Let $g_i$, $i=1,2$ be the
permutation, arising from the change $0\leftrightarrow 1$ on the
$i$-th coordinate of $X_m$. Obviously, $g_1,g_2$ satisfy to the
condition of the corollary.

Let $k=3$. Denote $g_1=(1,2)(2,3),g_2=(4,3)(3,2)\in S(2^2)$. For any
$m\geqslant 2$ one can consider $g_1,g_2$ as elements of $S(2^m)$,
acting on first two coordinates only. Obviously, these elements
satisfy to the conditions of the corollary.

 Let $k>4$ be odd. Put $M=k$. For any $m\geqslant k$ the number $2^m$
can be represented as $2^m=(2k-4)l+(2k-2)r$ with $r,l$ nonnegative
integer. Divide the set of $2^m$ elements onto $l$ subsets of $2k-4$
elements and $r$ subsets of $2k-2$ elements. Organize the
permutations $g_1,g_2$ on each subset using the lemma \ref{append
g1}. \\{\it Proof of the proposition \ref{append si}.} Let $s\in
S(2^n)$. For $x\in X_n$ denote $ord(x)$ the number of elements in
the trajectory of $x$. That is, $ord(x)$ is the length of the cycle
of $s$, containing $x$. Let $m$ be the maximal of $M(ord(x)),x\in
S(2^n)$ (see the corollary \ref{append Mk}). For any $x$, putting
$k=ord(x)$, let $g^{(k)}_i,i=1,2$ be the permutations from $S(2^m)$,
which satisfy to the conditions of the corollary \ref{append Mk}.
Let $r\in \mathbb{N}$. For $a=(a_1,a_2,\ldots,a_r),a_i\in
\{1,2\},k\in \{ord(x):x\in X_n\}$ denote $g_a^{(k)}\in
S(2^{mr}),s_a\in S(2^{n+mr})$ as follows:
\begin{eqnarray*}g^{(k)}_a(y_1,y_2,\ldots,y_r)=
\left(g^{(k)}_{a_1}(y_1),g^{(k)}_{a_2}(y_2),\ldots,g^{(k)}_{a_r}(y_r)\right),\\
s_a(x,y)=\left\{\begin{array}{ll}
(x,y),&\text{if }x\in Fix(s),\\
\left(s(x),g^{(ord(x))}_a(y)\right),&\text{otherwise},\end{array}\right.
\end{eqnarray*} where $x\in X_n,y_i\in X_m,y\in X_{mr}$. Note, that each
$g_a^{(k)}$ has cycles of length $k$ and $1$ only. It follows, that
$s_a$ is conjugate to $s$ for each $a$. Moreover, for any $a, b\in
\{1,2\}^r$ one has:
\begin{eqnarray*}
s_as_b^{-1}(x,y)=\left\{\begin{array}{ll}
(x,y),&\text{if }x\in Fix(s),\\
\left(x,g^{(ord(x))}_a\left(g^{(ord(x))}_b\right)^{-1}(y)\right),&\text{otherwise}.\end{array}\right.
\end{eqnarray*} Therefore, by the definition of $g_i^{(k)}$ (see the
corollary \ref{append Mk}), $Fix\left(s_as_b^{-1}\right)=Fix(s)$ and
$s_as_b^{-1}$ has only even cycles and fixed elements for $a\neq b$.
Thus, $s_a$ satisfy to the conditions of the proposition \ref{append
si}.
\end{proof}

{}

Artem Dudko, University of Toronto, artem.dudko@utoronto.ca

\end{document}